\documentclass[12pt,a4paper]{amsart}
\usepackage[top=1.13in, bottom=1.13in, left=1.15in, right=1.15in]{geometry}
\usepackage[T1]{fontenc}
\usepackage{latexsym}
\usepackage{amsmath}
\usepackage{shuffle}
\usepackage{amssymb, amscd, amsthm, mathrsfs}
\usepackage[all]{xy}
\usepackage[dvips]{graphicx}

\numberwithin{equation}{section}

\newtheorem{definition}{Definition}
\newtheorem{proposition}{Proposition}
\newtheorem{theorem}{Theorem}
\newtheorem{lemma}{Lemma}
\newtheorem{corollary}{Corollary}
\newtheorem{remark}{Remark}
\newtheorem{example}{Example}

\newtheorem{problem}{Problem}

\newcommand{\g}{\mathfrak{g}}

\newcommand{\D}{\displaystyle}
\newcommand{\ra}{\rightarrow}

\newcommand{\vp}{\varphi}

\newcommand{\ts}{\otimes}
\newcommand{\s}{\sigma}
\newcommand{\hh}{\hbar}

\newcommand{\mc}{\mathbb{K}}

\newcommand{\wt}{\widetilde}

\newcommand{\GVB}{\operatorname*{GVB}}
\newcommand{\VB}{\operatorname*{VB}}
\newcommand{\VH}{\operatorname*{VH}}
\newcommand{\GL}{\operatorname*{GL}}
\newcommand{\id}{\operatorname*{id}}
\newcommand{\End}{\operatorname*{End}}

\begin{document}

\title[Generalized virtual braid groups]{Generalized virtual braid groups, quasi-shuffle product and quantum groups}

\author{Xin FANG}
\address{Institut des Hautes \'Etudes Scientifiques, 35, route de Chartes, Bures-sur-Yvette 91440, France.}
\email{xinfang.math@gmail.com}
\maketitle 

\def\abstractname{Abstract}
\begin{abstract}
We introduce in this paper the generalized virtual braid group on $n$ strands $\GVB_n$, generalizing simultaneously the braid groups and their virtual versions. A Mastumoto-Tits type section lifting shuffles in a symmetric group $\mathfrak{S}_n$ to the monoid associated to $\GVB_n$ is constructed, which is then applied to characterize the quantum quasi-shuffle product. A family of representations of $\GVB_n$ is constructed using quantum groups.
\end{abstract}

\maketitle

\section{Introduction}

Quasi-shuffle product appeared firstly in the work of Newman and Radford \cite{NR79} in aim of constructing the cofree irreducible Hopf algebra on an algebra. This structure was rediscovered by Hoffman \cite{Hof00} to characterize the product of multiple zeta values (MZVs) which simultaneously encodes the product of quasi-symmetric functions, Rota-Baxter algebras (\cite{EG06}, \cite{GK00}) and commutative TriDendriform algebras \cite{Lod}.
\par
The quasi-shuffle product is quantized by Jian, Rosso and Zhang \cite{JRZ11} where a supplementary braid structure on the algebra appearing in the framework of Newman and Radford is considered. Recently, in a joint work with Rosso \cite{FR12}, the whole quantum group is realized using a properly chosen quantum quasi-shuffle structure.
\par
These (quantum) quasi-shuffle products are initially defined either in a functorial way using the universal property or by inductive formulas. Their combinatorial natures are studied by Guo and Keigher \cite{GK00} by introducing the mixable shuffles, which is then generalized to the quantum case by Jian \cite{J13}.
\par
One of the goals of this paper is to reveal symmetries behind the quantum quasi-shuffle structure by giving a representation-theoretical point of view on these products, which is more natural and computable than mixable shuffles. To archieve this goal, we need to introduce the generalized virtual braid groups, generalizing simultaneously the braid group and its virtual version \cite{Kauf}. After that we construct a Matsumoto-Tits type section (an analogue of the lift of permutations to the braid group) lifting $(p,q)$-shuffles from the symmetric group to the generalized virtual braid group which encodes the quantum quasi-shuffle product. Since the quasi-shuffle product is a particular case of its quantization, it is characterized by lifting shuffles to the virtual braid group. A main ingredient in constructing these generalized Matsumoto-Tits sections is the bubble decomposition introduced in Section \ref{Sec3}.
\par
Moreover, a family of representations of the generalized virtual braid groups is constructed using quantum groups and it is hoped to establish some quantum invariants for virtual links. Some other perspectives are discussed in the last section.

\section{Generalized virtual braid groups}\label{Sec1}

Let $\mc$ be a field of characteristic $0$. All vector spaces, algebras and tensor products are over $\mc$ if they are not specified otherwise.

\begin{definition}\label{Def:GVB}
The generalized virtual braid group on $n$ strands $\GVB_n$ is the group generated by $\s_i$, $\xi_i$ for $1\leq i\leq n-1$ which satisfy the following relations:
\begin{enumerate}
\item[(1).] for $|i-j|>1$, $
\s_i\s_j=\s_j\s_i$, $\s_i\xi_j=\xi_j\s_i$, $\xi_i\xi_j=\xi_j\xi_i$;
\item[(2).] for $1\leq i\leq n-2$, $
\s_i\s_{i+1}\s_i=\s_{i+1}\s_i\s_{i+1}$, $\xi_i\xi_{i+1}\xi_i=\xi_{i+1}\xi_i\xi_{i+1}$,
$\xi_i\s_{i+1}\s_i=\s_{i+1}\s_i\xi_{i+1}$, $\xi_{i+1}\s_i\s_{i+1}=\s_{i}\s_{i+1}\xi_i$.
\end{enumerate}
\end{definition}

The group $\GVB_n$ can be looked as a double braided group where the two braidings are compatible.
\par
Let $\mathfrak{S}_n$, $\mathfrak{B}_n$ and $\VB_n$ denote the symmetric group, the braid group and the virtual braid group \cite{Kauf} with $n$ letters or on $n$ strands, respectively. The monoids associated to $\GVB_n$, $\VB_n$ and $\mathfrak{B}_n$ will be denoted by $\GVB_n^+$, $\VB_n^+$ and $\mathfrak{B}_n^+$.
\par
The notion of the generalized virtual braid group enlarges that of the braided group and its virtual version. Relations between them can be explained in the following commutative diagram:
\[
\xymatrix{\GVB_n \ar[r]^-\alpha \ar[d]^-\gamma & \mathfrak{B}_n\ar[d]^-\beta \\ \VB_n \ar[r]^-\delta & \mathfrak{S}_n,}
\]
where $\alpha$ (resp. $\beta$; $\gamma$; $\delta$) is the quotient by the normal sub-group generated by $\xi_i$ (resp. $\s_i^2$; $\s_i^2$; $\xi_i$) for $1\leq i\leq n-1$.

\section{Generalized Matsumoto-Tits section}\label{Sec3}

\subsection{Bubble decomposition}

The name for this decomposition of permutations in symmetric groups arises from the bubble sort (see \cite{TAOCP3}, Section 5.2.2) which is inefficient in practical use but has interesting theoretical applications. 
\par
The symmetric group  $\mathfrak{S}_n$ with generators $s_i=(i,i+1)$ acts on $n$ sites by permuting their positions.
\par
We fix an integer $n\geq 2$ and $p,q\in\mathbb{N}$ such that $p+q=n$. A permutation $\s\in\mathfrak{S}_n$ is called a $(p,q)$-shuffle if $\s^{-1}(1)<\cdots<\s^{-1}(p)$ and $\s^{-1}(p+1)<\cdots<\s^{-1}(p+q)$; the set of all $(p,q)$-shuffles in $\mathfrak{S}_n$ will be denoted by $\mathfrak{S}_{p,q}$.
\par
We fix $p,q$ as above. There exists a bijection $\mathfrak{S}_n\overset{\sim}{\longrightarrow} \mathfrak{S}_{p,q}\cdot (\mathfrak{S}_p\times\mathfrak{S}_q)$ where $\mathfrak{S}_p$ acts on the first $p$ positions and $\mathfrak{S}_q$ acts on the last $q$ positions. Iterating this bijection results the following  bubble decomposition of symmetric groups (a set theoretical bijection):
$$\mathfrak{S}_n\overset{\sim}{\longrightarrow}\mathfrak{S}_{n-1,1}\times \mathfrak{S}_{n-2,1}\times\cdots\times \mathfrak{S}_{2,1}\times \mathfrak{S}_{1,1}$$
where $\mathfrak{S}_{k,1}$ is embedded into $\mathfrak{S}_n$ at the first  $k+1$ positions. 
\par
For $\s\in\mathfrak{S}_n$, we let $(\s^{(n-1)},\s^{(n-2)},\cdots,\s^{(1)})$ denote its bubble decomposition according to the above bijection where if $\s^{(k)}\neq e$, we write
$$\s^{(k)}=s_{t_k(\s)}s_{t_k(\s)+1}\cdots s_k\in\mathfrak{S}_{k,1}$$
for $1\leq t_k(\s)\leq k$ (if $\s^{(k)}=e$, we demand $t_k(\s)=0$).

\begin{remark}
It is easy to show that $l(\s)=l(\s^{(1)})+\cdots+l(\s^{(n-1)})$, hence the bubble decomposition gives a reduced expression of $\s\in\mathfrak{S}_n$.
\end{remark}

We establish several properties of $t_k(\s)$ when $\s$ is a $(p,q)$-shuffle, which will be useful in the proof of Theorem \ref{thm:main}.

\begin{lemma}\label{Lem1}
Let $\s\in\mathfrak{S}_{p,q}$ be a $(p,q)$-shuffle. Then for any $1\leq k<p$, $t_k(\s)=0$.
\end{lemma}

\begin{proof}
Let $n=p+q$. We prove by induction on $n$ that for any $p$, any $\s\in\mathfrak{S}_{p,q}$ and $1\leq k<p$, $t_k(\s)=0$.
\par
Start by the case $n=3$, the only non-trivial case is $p=2$ and $q=1$. In this case we need to show $t_1(\s)=0$ for any $\s\in\mathfrak{S}_{2,1}$. This is clear since $\mathfrak{S}_{2,1}=\{e,s_2,s_1s_2\}$.
\par
For general $n$, we consider the following well-known decomposition of $\mathfrak{S}_{p,q}$:
\begin{equation}\label{eqn1}
\mathfrak{S}_{p,q}=\mathfrak{S}^R_{p-1,q}\,\sqcup\,\mathfrak{S}_{p,q-1}^Rs_1\cdots s_p
\end{equation}
where 
\begin{enumerate}
\item $\mathfrak{S}_{p-1,q}^R$ and $\mathfrak{S}_{p,q-1}^R$ are images of $\mathfrak{S}_{p-1,q}$ and $\mathfrak{S}_{p,q-1}$ in $\mathfrak{S}_{p,q}$ under the inclusion $s_k\mapsto s_{k+1}$;
\item the set $\mathfrak{S}_{p-1,q}^Rs_1\cdots s_p$ is defined by $\{ss_1\cdots s_p|\ s\in\mathfrak{S}_{p-1,q}^R\}$.
\end{enumerate}
Taking $\s\in\mathfrak{S}_{p,q}$, there are two cases:
\begin{enumerate}
\item $\s\in\mathfrak{S}_{p-1,q}^R$: by induction, for any $2\leq k<p$, $t_k(\s)=0$. It remains to show that $t_1(\s)=0$: this is a consequence of $\s\in\mathfrak{S}_{p-1,q}^R$.
\item $\s\in\mathfrak{S}_{p,q-1}^Rs_1\cdots s_p$: we write $\s=ss_1\cdots s_p$ for some $s\in\mathfrak{S}_{p,q-1}^R$. By induction, for any $2\leq k<p+1$, $t_k(s)=0$, therefore for any $2\leq k<p$, $t_k(ss_1\cdots s_p)=t_k(s)=0$ (in fact, $t_p(\s)=1$). It is clear that $t_1(\s)=0$ since $s\in\mathfrak{S}_{p,q-1}^R$ acts on the last $p+q-1$ positions.
\end{enumerate}
\end{proof}

Therefore the bubble decomposition of $\s\in\mathfrak{S}_{p,q}$ admits the form $\s=\s^{(n-1)}\cdots\s^{(p)}$.
\par
The proof of this lemma gives an effective algorithm to compute the bubble decomposition of $(p,q)$-shuffles.
\newline
\par
\noindent
\textbf{Algorithm.} We use the decomposition (\ref{eqn1}) in the proof of Lemma \ref{Lem1} to give a recursive algorithm to compute the bubble decomposition. Suppose that the bubble decompositions of elements in $\mathfrak{S}_{p-1,q}$ and $\mathfrak{S}_{p,q-1}$ are known. For a fixed $\s\in\mathfrak{S}_{p,q}$, there are two cases:
\begin{enumerate}
\item $\s\in\mathfrak{S}_{p-1,q}^R$: according to the lemma, we may suppose that $\s=\s^{(n-1)}\cdots\s^{(p)}$ is the bubble decomposition of $\s$ in $\mathfrak{S}_{p-1,q}^R$. This is also the bubble decomposition of $\s$ in $\mathfrak{S}_{p,q}$.
\item $\s\in\mathfrak{S}_{p,q-1}^Rs_1\cdots s_p$: we write $\s=ss_1\cdots s_p$, by Lemma \ref{Lem1} again, we suppose that the bubble decomposition of $s$ in $\mathfrak{S}_{p,q-1}^R$ reads $s=s^{(n-1)}\cdots s^{(p+1)}$, we denote $s^{(p)}=s_1\cdots s_p$, then $\s=s^{(n-1)}\cdots s^{(p+1)}s^{(p)}$ is the bubble decomposition of $\s$.
\end{enumerate}

This algorithm allows us to prove the following lemma.

\begin{lemma}\label{Lem2}
Let $\s\in\mathfrak{S}_{p,q}$ be a $(p,q)$-shuffle and $p<s\leq p+q-1$. Then $t_s(\s)>s-p$.
\end{lemma}

\begin{proof}
We apply induction on $p+q=n$. The case $n=2$ is clear.
\par
For a general $n$ and $\s\in\mathfrak{S}_{p,q}$, we need to tackle the following two cases:
\begin{enumerate}
\item $\s\in\mathfrak{S}_{p-1,q}^R$: by induction, for any $p-1<s\leq p+q-2$, $t_{s+1}(\s)>s-p+1$, therefore $t_s(\s)>s-p$ for any $p<s\leq p+q-1$.
\item $\s=\s_0s_1\cdots s_p\in\mathfrak{S}_{p,q-1}^Rs_1\cdots s_p$: by induction, for any $p+1<s\leq p+q-1$, $t_s(\s_0)>s-p$. According to the second case of the above algorithm, $t_s(\s)>s-p$ for any $p<s\leq p+q-1$.
\end{enumerate}
\end{proof}

The following evident corollary is useful.

\begin{corollary}\label{Cor1}
Let $\s\in\mathfrak{S}_{p,q}$ be a $(p,q)$-shuffle. If $t_p(\s)\neq 1$, then for any $p\leq s\leq p+q-1$, $t_s(\s)\neq 1$.
\end{corollary}

\subsection{Generalized Matsumoto-Tits section}
The numbers $n,p,q$ are fixed as in the last subsection.
\par
We define the map $Q_{p,q}:\mathfrak{S}_{p,q}\ra\mc[\GVB_n^+]$ by
$$\s\mapsto\prod_{k=n,\s^{(k)}\neq e}^{k=1}(\s_{t_k}+(1-\delta_{t_k+1,t_{k+1}})\xi_{t_k})\s_{t_k+1}\cdots \s_k,$$
where $t_k=t_k(\s)$ and the product is executed in a descending way. These maps are called generalized Matsumoto-Tits section.

\begin{example}\label{Ex:1}
We take $n=4$ and $p=q=2$. The set of $(2,2)$-shuffles $\mathfrak{S}_{2,2}$ is given by $\{e, s_2, s_3s_2, s_1s_2, s_3s_1s_2, s_2s_3s_1s_2\}$. 
According to the bubble decomposition,
$$e\mapsto (e,e,e),\ \ s_2\mapsto (e,s_2,e),\ \ s_3s_2\mapsto (s_3,s_2,e),\ \ s_1s_2\mapsto (e,s_1s_2,e),$$
$$s_3s_1s_2\mapsto (s_3,s_1s_2,e),\ \ s_2s_3s_1s_2\mapsto (s_2s_3,s_1s_2,e).$$
Their images under the map $Q_{2,2}$ are:
$$Q_{2,2}(e)=e,\ \  Q_{2,2}(s_2)=\s_2+\xi_2,\ \  Q_{2,2}(s_3s_2)=(\s_3+\xi_3)\s_2,\ \ Q_{2,2}(s_1s_2)=(\s_1+\xi_1)\s_2,$$
$$Q_{2,2}(s_3s_1s_2)=(\s_3+\xi_3)(\s_1+\xi_1)\s_2,\ \ Q_{2,2}(s_2s_3s_1s_2)=(\s_2+\xi_2)\s_3\s_1\s_2.$$
\end{example}

Let $\wt{\gamma}:\mc[\GVB_n^+]\ra\mc[\VB_n^+]$ (resp. $\wt{\alpha}:\mc[\GVB_n^+]\ra\mc[\mathfrak{B}_n^+]$) be the quotient by the ideal generated by $\s_i^2-1$ (resp. $\xi_i$) for $1\leq i\leq n-1$. Their compositions with $Q_{p,q}$ give two maps 
$$V_{p,q}=\wt{\gamma}\circ Q_{p,q}:\mathfrak{S}_{p,q}\ra\mc[{\VB}_n^+],\ \ T_{p,q}=\wt{\alpha}\circ Q_{p,q}:\mathfrak{S}_{p,q}\ra\mc[\mathfrak{B}_n^+],$$
where $T_{p,q}$ coincides with the Matsumoto-Tits section $\mathfrak{S}_n\ra\mathfrak{B}_n^+$ (Theorem 4.2 in \cite{KT}). In the above example, $T_{2,2}(s_3s_1s_2)=\s_3\s_1\s_2$.
\par
These constructions are summarized in the following commutative diagram:
\[
\xymatrix{\mc[\mathfrak{B}_n^+] & \mc[{\GVB}_n^+]\ar[l]_-{\wt{\alpha}} \ar[d]^-{\wt{\gamma}}\\ \mathfrak{S}_{p,q}\ar[ur]^-{Q_{p,q}}\ar[u]^-{T_{p,q}} \ar[r]_-{V_{p,q}} & \mc[{\VB}_n^+]}.
\]

\section{Application: quantum quasi-shuffle product}

We will characterize in this section the quantum quasi-shuffle product using the generalized Matsumoto-Tits section.

\subsection{Braided algebra and quantum quasi-shuffle product}

The construction of quantum quasi-shuffle algebra can be viewed as a functor from the category of braided algebras to the category of braided Hopf algebras.

\begin{definition}
A braided algebra is a triple $(V,m,\s)$ where $V$ is a vector space, $m:V\ts V\ra V$ is an associative operation, $\s:V\ts V\ra V\ts V$ is a linear map such that on $V\ts V\ts V$, the following relations are verified:
$$(\s\ts\id)(\id\ts\s)(\s\ts\id)=(\id\ts\s)(\s\ts\id)(\id\ts\s),$$
$$(m\ts\id)(\id\ts\s)(\s\ts\id)=\s(m\ts\id),$$
$$(\id\ts m)(\s\ts\id)(\id\ts\s)=\s(\id\ts m).$$
\end{definition}

\begin{remark}
In the definition, $\s$ is not required to be invertible and the algebra $(V,m)$ is not required to be unitary.
\end{remark}

We fix a braided algebra $(V,m,\s)$ and let $T(V)=\bigoplus_{n\geq 0}V^{\ts n}$ to be the tensor vector space built on $V$. Each $V^{\ts n}$ admits a $\mathfrak{B}_n$-module structure through
$$\s_i\mapsto {\id}^{\ts(i-1)}\ts \s\ts {\id}^{\ts (n-i-1)}.$$
We denote $\beta_{p,1}=T_{p,1}(s_1\cdots s_{p-1}s_p)\in\mathfrak{B}_{p+1}^+$.

\begin{definition}[\cite{JRZ11}]\label{Def:qqs}
The quantum quasi-shuffle product is a family of operations $\ast_{p,q}:V^{\ts p}\underline{\ts} V^{\ts q}\ra T(V)$ defined by induction as follows: for $v_1,\cdots,v_{p+q}\in V$,
\begin{eqnarray*}
& &(v_1\ts\cdots\ts v_p)\ast_{p,q}(v_{p+1}\ts\cdots\ts v_{p+q})\\
&=& v_1\ts ((v_2\ts\cdots\ts v_p)\ast_{p-1,q}(v_{p+1}\ts\cdots\ts v_{p+q}))+\\
& &+(\id\ts\ast_{p,q-1})(\beta_{p,1}\ts{\id}^{\ts (q-1)})(v_1\ts\cdots\ts v_{p+q})+
\\& &+ (m\ts\ast_{p-1,q-1})(\id\ts\beta_{p-1,1}\ts{\id}^{\ts (q-1)})(v_1\ts\cdots\ts v_{p+q});
\end{eqnarray*}
when $p=0$ ou $q=0$, the operation is given by the scalar multiplication.
\end{definition}

It is proved in \cite{JRZ11} that the tensor space $T(V)$, endowed with this family of operations, is an associative algebra. It is called the quantum quasi-shuffle algebra and will be denoted by $T_{\s,m}(V)$.

\begin{example}\label{Ex:2}
\begin{enumerate}
\item When $p=q=1$, $v\ast_{1,1} w=(\id+\s+m)(v\ts w)\in V\oplus V\ts V$.
\item When $p=q=2$, the product $(v_1\ts v_2)\ast_{2,2}(v_3\ts v_4)$ is given by the action of
$$\id+\s_2+\s_3\s_2+m_3\s_2+\s_1\s_2+\s_3\s_1\s_2+m_3\s_1\s_2+\s_2\s_3\s_1\s_2+m_2\s_3\s_1\s_2+m_1\s_2+m_1\s_3\s_2+m_2m_1\s_2$$
on $v_1\ts v_2\ts v_3\ts v_4$ where $\s_i$ and $m_i$ signify that $\s$ and $m$ act on the positions $i$ and $i+1$.
\end{enumerate}
\end{example}

When $m=0$, the quantum quasi-shuffle product reduces to the quantum shuffle product (\cite{Ros95}, \cite{Ros98}). When $\s^2={\id}_{V\ts V}$, the quantum quasi-shuffle product gives the quasi-shuffle product (\cite{NR79}), if moreover the multiplication $m$ is commutative, this product is rediscovered by Hoffman in \cite{Hof00} to characterize the product of multiple zeta values (MZVs).
\par
The objective of this section is to realize these three products on $T(V)$ in a uniform framework.

\subsection{A representation-theoretical point of view on braided algebras}
In \cite{L12A} Lebed realized the associativity of an algebra as an exotic (pre-)braiding, which will be recalled in the following proposition. This construction allows us to reveal the nature of a braided algebra in the spirit of operads. 
\par
Let $V$ be a vector space with a specified element $\bold{1}\in V$, $m:V\ts V\ra V$ and $\nu:\mc\ra V$ be two linear maps where $\nu$ sends $1$ to $\bold{1}\in V$. We define a linear map $\s^m:V\ts V\ra V\ts V$ by $v\ts w\mapsto \bold{1}\ts m(v\ts w)$.

\begin{proposition}[\cite{L12A}]\label{Prop:Vic}
Suppose that for any $v\in V$, $m(v\ts\bold{1})=v$. Then $m$ is associative if and only if $\mathfrak{B}_3^+\ra\End(V^{\ts 3})$, $\s_1\mapsto \s^m\ts \id$, $\s_2\mapsto\id\ts\s^m$ defines a representation of $\mathfrak{B}_3^+$.
\end{proposition}

When a supplementary braiding structure is under consideration, this proposition can be generalized:

\begin{proposition}\label{Prop:GVB}
We suppose that for any $v\in V$, $m(v\ts\bold{1})=v$, $m(\bold{1}\ts v)=v$, $\s(\bold{1}\ts v)=v\ts\bold{1}$ and $\s(v\ts\bold{1})=\bold{1}\ts v$. Then $(V,\s,m)$ is a braided algebra if and only if $\GVB_3^+\ra \End(V^{\ts 3})$, $\s_1\mapsto \s\ts\id$, $\s_2\mapsto\id\ts\s$, $\xi_1\mapsto \s^m\ts\id$, $\xi_2\mapsto\id\ts\s^m$ defines a representation of $\GVB_3^+$.
\end{proposition}

\begin{proof}
It suffices to show that relation in Definition \ref{Def:GVB} (2) are verified. The second relation therein holds by Proposition \ref{Prop:Vic} and the other relations are exactly those in the definition of a braided algebra.
\end{proof}

\subsection{Quantum quasi-shuffle product revisited}

Let $(V,\s,m)$ be a braided algebra with $\nu:\mc\ra V$ sending $1$ to $\bold{1}\in V$ which satisfies 
$$m(v\ts\bold{1})=v,\ \ m(\bold{1}\ts v)=v,\ \  \s(\bold{1}\ts v)=v\ts\bold{1},\ \ \s(v\ts\bold{1})=\bold{1}\ts v.$$

\begin{remark}
The requirement of the existence of this specified element $\bold{1}\in V$ causes no trouble: given a braided algebra $(V,\s,m)$, enlarging $V$ by $\widetilde{V}=\mc\bold{1}\oplus V$ and naturally extending $\s$ and $m$ solves this problem.
\end{remark}

According to Proposition \ref{Prop:GVB}, for any $n\geq 1$, $V^{\ts n}$ admits a representation of $\GVB_n^+$ by
$$\s_i\mapsto {\id}^{\ts (n-1)}\ts \s\ts{\id}^{\ts (n-i-1)},\ \ \xi_i\mapsto {\id}^{\ts (n-1)}\ts \s^m\ts{\id}^{\ts (n-i-1)}.$$
\par
Let $\widehat{V}$ be the linear complement of $\mc\bold{1}$ in $V$.

\begin{definition}
\begin{enumerate}
\item A pure tensor $v_1\ts\cdots\ts v_k\in T(V)$ is called deletable if for any $ i=1,\cdots,k$, either $v_i=\bold{1}$ or $v_i\in\widehat{V}$.
\item The deleting operator $\bold{D}:T(V)\ra T(V)$ is defined as the linear map deleting all $\bold{1}$ in a deletable pure tensor.
\end{enumerate}
\end{definition}

For example, $\bold{D}(\bold{1}\ts v\ts w\ts \bold{1}\ts \bold{1}\ts u\ts \bold{1})=v\ts w\ts u$ where $u,v,w$ are elements in $\widehat{V}$.

\begin{theorem}\label{thm:main}
Let $\ast$ denote the quantum quasi-shuffle product on $T(V)$. Then for any $v_1,\cdots,v_{p+q}\in \widehat{V}$,
$$(v_1\ts\cdots\ts v_p)\ast_{p,q}(v_{p+1}\ts\cdots\ts v_{p+q})=\bold{D}\circ \sum_{s\in\mathfrak{S}_{p,q}}Q_{p,q}(s)(v_1\ts\cdots\ts v_{p+q}).$$ 
\end{theorem}

This theorem will be proved in the next subsection.

\begin{example}
The quantum quasi-shuffle product $(v_1\ts v_2)\ast_{2,2}(v_3\ts v_4)$ is computed in Example \ref{Ex:2}, it coincides with the lifting of $(2,2)$-shuffles through the generalized Matsumoto-Tits section in Example \ref{Ex:1}.
\end{example}

\begin{remark}
In \cite{GK00} and \cite{J13}, combinatorial formulas on the quasi-shuffle and quantum quasi-shuffle products are proved with the help of mixable shuffles. The combinatorial nature of Theorem \ref{thm:main} differs from those in the papers cited above. Our construction on the lifting of shuffles relies on the bubble decomposition, which is more natural and efficiently computable.
\par
We explain in the following example the difference between our lifting and the mixable shuffles. Consider the $(2,2)$-shuffles: in the context of mixable shuffles, the element $m_1\s_3\s_2$ corresponds to the $(2,2)$-shuffle $s_3s_2$; in our framework, as shown in Example \ref{Ex:1}, it arises from the $(2,2)$-shuffle $s_3s_1s_2$.
\end{remark}

As the quasi-shuffle product is a particular case ($\s^2=\id$) of the quantum quasi-shuffle product, composing with the morphism of monoid $\wt{\gamma}$ gives:

\begin{corollary}
Let $V$ be an associative algebra and $\bullet$ be the quasi-shuffle product on $T(V)$ (\cite{NR79}, \cite{Hof00}). Then for $v_1,\cdots,v_{p+q}\in V$ different from $\bold{1}$,
$$(v_1\ts\cdots\ts v_p)\bullet_{p,q}(v_{p+1}\ts\cdots\ts v_{p+q})=\bold{D}\circ \sum_{s\in\mathfrak{S}_{p,q}}V_{p,q}(s)(v_1\ts\cdots\ts v_{p+q}).$$
\end{corollary}

These formulas generalize those for quantum shuffle product $\diamondsuit$ (\cite{Ros95}, \cite{Ros98}): for a braided vector space $(V,\s)$,
$$(v_1\ts\cdots\ts v_p)\diamondsuit_{p,q}(v_{p+1}\ts\cdots\ts v_{p+q})=\sum_{s\in\mathfrak{S}_{p,q}}T_{p,q}(s)(v_1\ts\cdots\ts v_{p+q}).$$

\subsection{Proof of Theorem \ref{thm:main}}\label{sec:proof}

We prove it by induction on $n=p+q$. The case $p+q=1$ is clear.
\par
By the inductive definition of the quantum quasi-shuffle product in Definition \ref{Def:qqs} and the induction hypothesis, it suffices to show that
$$\sum_{s\in\mathfrak{S}_{p,q}}Q_{p,q}(s)=\sum_{s\in\mathfrak{S}_{p-1,q}^R}Q_{p-1,q}(s)+\sum_{s\in\mathfrak{S}_{p,q-1}^R}Q_{p,q-1}(s)\s_1\cdots\s_p+\D \sum_{s\in\mathfrak{S}_{p-1,q-1}^R}Q_{p-1,q-1}(s)\xi_1\s_2\cdots\s_p,$$
where the notations $\mathfrak{S}_{p,q-1}^R$ and $\mathfrak{S}_{p-1,q}^R$ are explained in the proof of Lemma \ref{Lem1}. The set $\mathfrak{S}_{p-1,q-1}^R$ is the image of $\mathfrak{S}_{p-1,q-1}$ in $\mathfrak{S}_{p,q}$ under the inclusion $s_i\mapsto s_{i+2}$ for $1\leq i\leq p+q-2$. 
\par
We divide $\mathfrak{S}_{p,q}$ into three disjoint subsets
$$\mathfrak{S}_{p,q}=S_1\sqcup S_2\sqcup S_3$$
where
$$S_1=\{\s\in\mathfrak{S}_{p,q}|\ t_p(\s)\neq 1\},\ \ S_2=\{\s\in\mathfrak{S}_{p,q}|\ t_p(\s)=1,\ t_{p+1}(\s)=2\},$$
$$S_3=\{\s\in\mathfrak{S}_{p,q}|\ t_p(\s)=1,\ t_{p+1}(\s)\neq 2\}.$$

We will in fact prove the following three identities, then taking the sum implies the formula desired:
\begin{equation}\label{eq3}
\sum_{\s\in S_1}Q_{p,q}(\s)=\sum_{\s\in\mathfrak{S}_{p-1,q}^R}Q_{p-1,q}(\s),
\end{equation}
\begin{equation}\label{eq4}
\sum_{\s\in S_2}Q_{p,q}(\s)=\sum_{\s\in\mathfrak{S}_{p,q-1}^R\backslash\mathfrak{S}_{p-1,q-1}^R}Q_{p,q-1}(\s^{(n-1)}\cdots \s^{(p+1)})\s_1\cdots\s_p,
\end{equation}
\begin{equation}\label{eq5}
\sum_{\s\in S_3}Q_{p,q}(\s)=\sum_{\s\in\mathfrak{S}_{p-1,q-1}^R}Q_{p-1,q-1}(\s^{(n-1)}\cdots \s^{(p+1)})(\s_1+\xi_1)\s_2\cdots\s_p.
\end{equation}

We start by giving some combinatorial explanations of these three sets.

\begin{enumerate}
\item We show that $S_1=\{\s\in\mathfrak{S}_{p,q}|\ \s(1)=1\}$. That is to say, for any $(p,q)$-shuffle $\s$, $\s(1)=1$ if and only if $t_p(\s)\neq 1$. According to Lemma \ref{Lem1} and Corollary \ref{Cor1}, if $t_p(\s)\neq 1$, then $s_1$ does not appear in a reduced expression of $\s$, from which $\s(1)=1$. On the other hand, if $\s(1)=1$, then $t_p(\s)\neq 1$ by Lemma \ref{Lem2}.
\item We prove that $S_2=(\mathfrak{S}_{p,q-1}^R\backslash\mathfrak{S}_{p-1,q-1}^R)s_1\cdots s_p$.
\par Let $\s\in S_2$, which means that $\s\in\mathfrak{S}_{p,q}$, $t_p(\s)=1$ and $t_{p+1}(\s)=2$. The condition $t_p(\s)=1$ implies that $\s^{(p)}=s_1\cdots s_p$. From the bubble decomposition and Lemma \ref{Lem1}, $\s=\s^{(n)}\cdots\s^{(p+1)}s_1\cdots s_p$. We show that $\s'=\s^{(n)}\cdots\s^{(p+1)}\in \mathfrak{S}_{p,q-1}^R\backslash\mathfrak{S}_{p-1,q-1}^R$: to show that it is not in $\mathfrak{S}_{p-1,q-1}^R$, whose all elements in $\mathfrak{S}_{p,q-1}$ preserving the position $2$, we consider $t_{p+1}(\s)$: if $t_{p+1}(\s)=2$, then $\s^{(p+1)}=s_2\cdots s_{p+1}$ will not preserve the position $2$; by Lemma \ref{Lem2}, so does $\s'$.
\par
We take $\s\in\mathfrak{S}_{p,q}$ with bubble decomposition $\s=\s^{(n-1)}\cdots \s^{(p+1)}s_1\cdots s_p$ where $\s'=\s^{(n-1)}\cdots \s^{(p+1)}\in\mathfrak{S}_{p,q-1}^R\backslash\mathfrak{S}_{p-1,q-1}^R$. Since $\s'\notin\mathfrak{S}_{p-1,q-1}^R$, $t_p(\s)=1$. Moreover, $\s'$ does not preserve the position $2$, then Lemma \ref{Lem2} forces $\s^{(p+1)}=s_2\cdots s_{p+1}$, from which $t_{p+1}(\s)=2$.
\item Finally we claim that $S_3=\mathfrak{S}_{p-1,q-1}^Rs_1\cdots s_p$. As $\mathfrak{S}_{p,q-1}^Rs_1\cdots s_p=\{\s\in\mathfrak{S}_{p,q}|\ t_p(\s)=1\}$, the desired identity is the complement of the one proved in (2).
\end{enumerate}

By definition, for $\s\in S_1$, $Q_{p,q}(\s)=Q_{p-1,q}(\s)$ and therefore (\ref{eq3}) holds. Let $\s\in S_2$. It admits the decomposition $\s=\s^{(n-1)}\cdots \s^{(p+1)}s_1\cdots s_p$, since $t_{p+1}(\s)=2$, in the definition of the generalized Matsumoto-Tits section, terms including the Kronecker $\delta$ vanish, it therefore gives
$$Q_{p,q}(\s)=Q_{p,q-1}(\s^{(n-1)}\cdots \s^{(p+1)})\s_1\cdots\s_p.$$
This proves (\ref{eq4}). Finally, for $\s\in S_3$, the same argument as above shows (\ref{eq5}). The proof of Theorem \ref{thm:main} terminates.

\section{Representations of generalized virtual braid groups and quantum groups}

We want to construct representations of the generalized virtual braid group via quantum groups: it requires us to investigate the other structure inside of a quantum group encoding the virtual crossings.

Let $V$ be a vector space. We let $P:V\ts V\ra V\ts V$ denote the usual flip defined by $P(v\ts w)=w\ts v$ for $v,w\in V$.

\subsection{Quasi-triangular Hopf algebras with a twist}

This structure we desired is firstly studied by Drinfeld in quasi-Hopf algebras and then by Reshetikhin \cite{Res} in the early stage of quantum groups.

\begin{definition}[\cite{Res}]\label{Def:AHQTT}
Let $H$ be a Hopf algebra and $R,F\in H\ts H$ be two invertible elements. The triple $(H,R,F)$ is called a quasi-triangular Hopf algebra with a twist (QTHAT) if
\begin{enumerate}
\item $(H,R)$ if a quasi-triangular Hopf algebra (QTHA) (\cite{Dri86});
\item $(\Delta\ts\id)(F)=F_{13}F_{23}$, $(\id\ts\Delta)(F)=F_{13}F_{12}$, $F_{12}F_{13}F_{23}=F_{23}F_{13}F_{12}$ where $F_{12}=F\ts 1$, $F_{23}=1\ts F$ et $F_{13}=\sum{a_i}\ts 1\ts b_i$ if $F=\sum a_i\ts b_i$. Such an $F$ is called a twist.
\end{enumerate}
If moreover $F$ satisfies $F_{21}F=1\ts 1$ where $F_{21}=\sum b_i\ts a_i$, the triple is called involutive.
\end{definition}

For example, any quasi-triangular Hopf algebra has a twist $F=1\ts 1$.
\par
In \cite{Res}, only involutive triples are considered. It is interesting to find non-involutive QTHATs.

The following lemma is well-known, whose proof is a standard exercise.

\begin{lemma}\label{Lem:QTHAT}
Let $(H,R,F)$ be a QTHAT. Then
$$R_{12}F_{13}F_{23}=F_{23}F_{13}R_{12},\ \ F_{12}F_{13}R_{23}=R_{23}F_{13}F_{12}.$$
\end{lemma}

The QTHATs allow us to construct tensor representations of the generalized virtual braid groups.

\begin{proposition}\label{Prop:QTHAT}
Let $(V,\rho)$ be a representation of a QTHAT $(H,R,F)$. We denote $\check{R}=P\circ(\rho\ts\rho)(R)$ and $\check{F}=P\circ (\rho\ts\rho)(F)$. Then 
$$\mu:{\GVB}_n\ra {\GL}(V^{\ts n}),$$
$$\s_i\mapsto{\id}^{\ts (i-1)}\ts \check{R}\ts {\id}^{\ts (n-i-1)},\ \ \xi_i\mapsto{\id}^{\ts (i-1)}\ts \check{F}\ts {\id}^{\ts (n-i-1)}$$
defines a representation of the group $\GVB_n$. Moreover, if the triple $(H,R,F)$ is involutive, $\mu$ factorizes through $\VB_n$ and gives a representation of the virtual braid group $\VB_n$.
\end{proposition}

\begin{proof}
This proposition is a consequence of the definition of a QTHAT and Lemma \ref{Lem:QTHAT}.
\end{proof}

\subsection{Example: quantum groups}

Let $\g$ be a finite dimensional semi-simple Lie algebra of rank $l$ over $\mathbb{C}$ and $U_\hh(\g)$ be the $\hh$-adic version of the quantum group associated to $\g$ (\cite{Dri86}) with generators $E_i$, $F_i$ and $H_i$ for $i=1,\cdots,l$.
\par
The Hopf algebra $U_\hh(\g)$ is quasi-triangular with the universal R-matrix $R$ (\cite{Dri86}). Reshetikhin constructed in \cite{Res} a non-trivial element $F\in U_\hh(\g)\ts U_\hh(\g)$:
$$F=\exp\left(\sum_{i<j}\vp_{ij}(H_i\ts H_j-H_j\ts H_i)\right)$$
such that $(U_\hh(\g),R,F)$ is an involutive QTHAT where $\vp_{ij}\in\mathbb{C}$ are constants. 
\par
The following theorem is a direct consequence of Proposition \ref{Prop:QTHAT}.

\begin{theorem}\label{Thm:GVB}
Let $W$ be a finite dimensional representation of the quantum group $U_\hh(\g)$. Then for any $n\geq 1$, $W^{\ts n}$ admits a $\GVB_n$-module structure which is factorizable through $\VB_n$.
\end{theorem}

\begin{remark}
The categorical interpretation of the virtual braids \cite{L12B} can be generalized to the case of this paper to obtain the notion of the double braided categories. The category of finite dimensional representations of the quantum group $U_\hh(\g)$ admits such a structure.
\end{remark}

\begin{problem}\label{Pb:1}
Theorem \ref{Thm:GVB} serves as the starting point of the quantum invariant theory of the virtual links. That is to say, we want to know whether the construction of Turaev \cite{V88} can be modified to fit for the virtual links.
\end{problem}

\begin{problem}
The virtual Hecke algebra $\VH_n(q)$ is a quotient of the generalized virtual braid group $\GVB_n$ by relations: $\xi_i^2=1$, $\s_i^2=(q-1)\s_i+q$ (equivalently, it is a quotient of the virtual braid group $\VB_n$ by the relation $\s_i^2=(q-1)\s_i+q$). Problem \ref{Pb:1} is equivalent to the study of the existence of a virtual Markov trace on $\VH_n(q)$. Moreover, it would be interesting to find the object having Schur-Weyl duality with $\VH_n(q)$, it is expected to be some object interpolating the classical $\mathfrak{gl}_n$ and quantum $\mathfrak{sl}_n$.
\end{problem}

\begin{problem}
As the universal R-matrix admits the uniqueness property. It is natural to ask for a classification of the invertible elements $F\in U_\hh(\g)\widehat{\ts} U_\hh(\g)$ satisfying the point (2) in Definition \ref{Def:AHQTT} and involutive.
\end{problem}

\section*{Acknowledgements}
I would like to thank Run-Qiang Jian for stimulating discussions. This work is partially supported by the National Natural Science Foundation of China (Grant No. 11201067).

\end{document}